\documentclass[12pt,a4paper]{amsart}
\usepackage{mathrsfs}

\newtheorem{theo+}              {Theorem}           [section]
\newtheorem{prop+}  [theo+]     {Proposition}
\newtheorem{coro+}  [theo+]     {Corollary}
\newtheorem{lemm+}  [theo+]     {Lemma}
\newtheorem{exam+}  [theo+]     {Example}
\newtheorem{rema+}  [theo+]     {Remark}
\newtheorem{defi+}  [theo+]     {Definition}

\newenvironment{theorem}{\begin{theo+}}{\end{theo+}}
\newenvironment{proposition}{\begin{prop+}}{\end{prop+}}
\newenvironment{corollary}{\begin{coro+}}{\end{coro+}}

\usepackage{amsthm}
\theoremstyle{plain} \theoremstyle{remark}
\newtheorem{remark}{Remark}
\newtheorem{example}{Example}

\newtheorem*{ack}{\bf Acknowledgments}
\renewcommand{\Bbb}{\mathbb}

\def \r{\mbox{${\mathbb R}$}}
\def\E{/\kern-1.0em \equiv }

\title{Some constructions of biharmonic maps and Chen's conjecture on biharmonic hypersurfaces}

\author{Ye-Lin Ou$^{*}$ }

\address{Department of
Mathematics,\newline\indent Texas A $\&$ M University-Commerce,
\newline\indent Commerce TX 75429,\newline\indent USA.\newline\indent
E-mail:yelin$\_$ou@tamu-commerce.edu}
\thanks{$^*$ Supported by Texas A $\&$ M University-Commerce
``Faculty Research Enhancement Project" (2010-11)}
\date{08/13/18}
\begin{document}
\title[Constructions of biharmonic maps $\&$ Chen's conjecture ]{Some constructions of biharmonic maps and Chen's
conjecture on biharmonic hypersurfaces}


\maketitle

\section*{Abstract}
\begin{quote}
{\footnotesize We give several construction methods and use them to
produce many examples of proper biharmonic maps including biharmonic
 tori of any dimension in Euclidean spheres (Theorem \ref{OTh},
Corollaries \ref{TS}, \ref{TS4}), biharmonic maps
between spheres (Theorem \ref{S}) and into spheres (Theorem
\ref{JI}) via orthogonal multiplications and eigenmaps. We also
study biharmonic graphs of maps, derive the equation for a function
whose graph is a biharmonic hypersurface in a Euclidean space,  and
give an equivalent formulation of Chen's conjecture on biharmonic
hypersurfaces by using the biahrmonic graph equation (Theorem
\ref{Conj}) which paves a way for analytic study of the conjecture.}
\end{quote}
\section{introduction}
In this paper, all manifolds, maps, and tensor fields are assumed to be smooth unless there is an otherwise statement.\\

A biharmonic map is a map $\varphi:(M, g)\longrightarrow (N, h)$
between Riemannian manifolds that is a critical point of the
bienergy functional
\begin{equation}\nonumber
E^{2}\left(\varphi,\Omega \right)= \frac{1}{2} {\int}_{\Omega}
\left|\tau(\varphi) \right|^{2}{\rm d}x
\end{equation}
for every compact subset $\Omega$ of $M$, where $\tau(\varphi)={\rm
Trace}_{g}\nabla {\rm d} \varphi$ is the tension field of $\varphi$.
The Euler-Lagrange equation of this functional gives the biharmonic
map equation (\cite{Ji})
\begin{equation}\label{BTF}
\tau^{2}(\varphi):={\rm
Trace}_{g}(\nabla^{\varphi}\nabla^{\varphi}-\nabla^{\varphi}_{\nabla^{M}})\tau(\varphi)
- {\rm Trace}_{g} R^{N}({\rm d}\varphi, \tau(\varphi)){\rm d}\varphi
=0,
\end{equation}
which states the fact that the map $\varphi$ is biharmonic if and
only if its bitension field $\tau^{2}(\varphi)$ vanishes
identically. In the above equation we have used $R^{N}$ to denote
the curvature operator of $(N, h)$ defined by
$$R^{N}(X,Y)Z=
[\nabla^{N}_{X},\nabla^{N}_{Y}]Z-\nabla^{N}_{[X,Y]}Z.$$ 

It is clear from (\ref{BTF}) that any harmonic map is biharmonic, so we
call those non-harmonic biharmonic maps {\bf proper biharmonic maps}.\\

A submanifold is {\bf biharmonic} if the isometric immersion
defining the submanifold is a biharmonic map. It is well known that
an isometric immersion is minimal if and only if it is harmonic. So
a minimal submanifold is trivially biharmonic and
we call those non-minimal biharmonic submanifolds {\bf proper biharmonic submanifolds}.\\

Examples of proper biharmonic maps are very difficult to find. Most
of the known examples come from the following families.
\begin{itemize}
\item [1.] {\bf Biharmonic isometric immersions:} {\bf (i)} The generalized Clifford
torus  $S^p(\frac{1}{\sqrt{2}})\times
S^{q}(\frac{1}{\sqrt{2}})\hookrightarrow S^n$ with $p\ne q, p+q=n$
(\cite{Ji3}); {\bf(ii)}  The hypersphere
$S^{n}(\frac{1}{\sqrt{2}})\hookrightarrow S^{n+1}$ (\cite{CMO1});
{\bf(iii)} Biharmonic cylinder : $S^1(\frac{1}{\sqrt{2}})\times
\mathbb{R}\hookrightarrow S^2\times \mathbb{R}$ (\cite{Ou5});
{\bf(iv)} The hyperplanes $z=k\;\;(k \;{\rm is \;
a\;constant}\;>-C)$ in the conformally flat space
$(\mathbb{R}^{m+1}_+,h=(\frac{z+C}{D})^2(\sum_{i=1}^m {\rm
d}x_i^{2}+{\rm d}z^{2}))$ (\cite{Ou5});
\item [2.] {\bf Biharmonic conformal immersions:} {\bf(i)} The inversion
in $3$-sphere:\; $\phi: \r^4\setminus \{0\}\longrightarrow \r^4$
with $\phi(x)=\frac{x\;}{|x|^2}$ (\cite{BK}). This is also the only
known biharmonic morphism which is not a harmonic morphisms
(\cite{LOU}); {\bf(ii)} Some proper biharmonic identity maps
\cite{LOU}: Euclidean-to-Hyperbolic space, $id: (H^4=\r^3\times
\r^+, ds^2)\longrightarrow (H^4,x_4^{-2}ds^2)$ and $id: (B^4=\{x\in
\r^4:|x|<1\}, ds^2)\longrightarrow (B^4,4(1-|x|^2)^{-2}ds^2)$;
Euclidean-to-spherical space, $id: (\r^4, ds^2)\longrightarrow
(S^4\setminus\{N\},4(1+|x|^2)^{-2}ds^2)$. {\bf(iii)} Conformal
immersions from Euclidean space into space forms \cite{Ou3}: $
\varphi : (\r^3\times \r^{+},\bar{g}=\delta_{ij}) \longrightarrow
(H^5=\r^4\times \r^{+},h=y_5^{-2}\delta_{\alpha\beta}) $
 with $\varphi(x_1,\ldots,x_4)=(1,x_1,\ldots,x_4)$ and $
\varphi : (\r^4,\bar{g}=\delta_{ij}) \longrightarrow (S^5\setminus
\{N\}\equiv \r^5,h=\frac{4\delta_{\alpha\beta}}{(1+|y|^2)^2}) $ with
$\varphi(u_1,\ldots,u_4)=(u_1,\ldots,u_4,0)$, where
$(u_1,\ldots,u_5)$ are conformal coordinates on $S^5\setminus
\{N\}\equiv \r^5$; {\bf(iv)} The conformal biharmonic surfaces in
$\r^3$ \cite{Ou3}: For $\lambda^2=\big(C_2e^{\pm
z/R}-C_1C_2^{-1}R^2e^{\mp z/R}\big)/2$ with constants $C_1, C_2$,
the maps $\phi:( D, g=\lambda^{-2}(R^2
d\theta^2+dz^2))\longrightarrow
(\r^3,d\sigma^2=d\rho^2+\rho^2\,d\theta^2+dz^2)$ with $\phi(\theta,
z)=(R, \theta, z)$ is a family of proper biharmonic conformal
immersions of  a cylinder of radius $R$ into Euclidean space $\r^3$,
where $D=\{ (\theta,z)\in (0,2\pi)\times \r: z\ne\pm \frac{R}{2}\ln
(C_1R^2/C_2^2)\}$.
\item[3.] {\bf Biharmonic horizontally weakly conformal maps:}
{\bf(i)} Hopf construction map $\phi:\r^2\times \r^2\longrightarrow
\r\times \r^2$, $\phi(z,w)=(|z|^2-|w|^2, 2z{\bar w})$ with the
standard Euclidean metric on the domain and certain conformally flat
metric on the target space (\cite{Oua}); {\bf(ii)} The map $\phi:
\r^3\longrightarrow (\r^2, {\bar g})$ with
$\phi(x_1,x_2,x_3)=(\sqrt{x_1^2+x_2^2\;}, x_3)$ and a certain
conformally flat metric ${\bar g}$ (\cite{Oua});
{\bf(iii)} Biharmonic Riemannian submersions \cite{LOU}:\\
$\phi : (\r^2 \times \r , dx^2 + dy^2 + \beta^2(x) dz^2) \to (\r^2 ,
dx^2 + dy^2),\; \phi(x,y,z) = (x,y)$ with $c_1 , c_2 \in \r^*$,
$\beta = c_2 \, e^{\int f(x) \, dx}$, and $f(x) = \frac{-c_1
(1+e^{c_1 x})}{1-e^{c_1 x}}$, is a family of biharmonic Riemannian
submersions;
\item[4.] {\bf Biharmonic maps which are neither conformal immersions nor
horizontally weakly conformal maps:} {\bf (i)} The composition of
the Hopf map and the inclusion:
$S^3(\frac{1}{\sqrt{2}})\longrightarrow
S^2(\frac{1}{\sqrt{2}})\hookrightarrow S^3$ (\cite{LOn1}), and more
generally, the pull-backs of proper biharmonic maps
$S^m(\frac{1}{\sqrt{2}})\hookrightarrow S^{m+1} (m=2, 4, 8)$ by the
Hopf fibrations $S^{2m-1}\longrightarrow S^m(\frac{1}{\sqrt{2}})$
(\cite{Ou2}); {\bf (ii)} The composition of a harmonic map $\psi$
from a compact manifold with constant energy density and the
inclusion map of a biharmonic hypersurface \cite{LOn1}: $ i\circ
\psi: (M,g)\longrightarrow S^n(\frac{1}{\sqrt{2}})\hookrightarrow
S^{n+1}$; {\bf (iii)} The axially symmetric biharmonic maps
\cite{BMO0}: $\phi: (0,\infty)_{t^2}\times S^{m-1}\longrightarrow
\r\times_{f^2}S^{n-1},\;\;\phi(t,x)=(\rho(t),\varphi_0(x))$, where
$\varphi_0$ is a constant map; $\phi: \r^m\setminus
\{0\}\longrightarrow \r^m\setminus \{0\}$ with $\phi(x)=x/|x|^{m-2},
\;m\ne 4$.
\end{itemize}

In this paper, we give several construction methods and use them to
produce many examples of proper biharmonic maps including biharmonic
 tori of any dimension in Euclidean spheres (Theorem \ref{OTh},
Corollaries \ref{TS}, \ref{TS4} and \ref{TS3}), biharmonic maps
between spheres (Theorem \ref{S}) and into spheres (Theorem
\ref{JI}) via orthogonal multiplications and eigenmaps. We also
study biharmonic graphs of maps, derive the equation for a function
whose graph is a biharmonic hypersurface in a Euclidean space,  and
give an equivalent formulation of Chen's conjecture on biharmonic
hypersurfaces by using the biharmonic graph equation (Theorem
\ref{Conj}) which reveals a similarity  to the well-known
Bernstein's conjecture on the existence of entire minimal graph. We
hope this will pave a way and attract more work and especially more
analytic study to Chen's conjecture on biharmonic hypersurfaces.\\

\section{Constructions of proper Biharmonic maps}

Let $(M_1\times M_2, g_1\times g_2)$ be the Riemannian product of
manifolds $(M_1^m,g_1)$ and $(M_2^n,g_2)$. For any map $F:(M_1\times
M_2, g_1\times g_2)\longrightarrow (N,h)$ with $F=F(x_1,x_2)$ we
have two families of maps $F_1, :(M_1^m,g_1)\longrightarrow (N,h)$
with $F_1(x_1)=F(x_1,x_2)$ for fixed $x_2\in M_2$ and $F_2,
:(M_2^n,g_2)\longrightarrow (N,h)$ with $F_2(x_2)=F(x_1,x_2)$ for
fixed $x_1\in M$. Then, we know (cf. \cite{EL1}, Section (4.15))
that $F$ is harmonic if it is harmonic with respect to each variable
separately, i.e., both $F_1$ and $F_2$ are harmonic. This
 can be generalized to the case of
biharmonic maps as

\begin{proposition}\label{ML}{\bf(Biharmonic maps from product
spaces)} {\bf(Biharmonic maps from product
spaces)} Let $F:(M_1\times M_2, g_1\times g_2)\longrightarrow
(N,h)$ with $F=F(x_1,x_2)$ and  $F_1, :(M_1^m,g_1)\longrightarrow (N,h)$ with
$F_1(x_1)=F(x_1,x_2)$ for fixed $x_2\in M_2$ and $F_2,
:(M_2^n,g_2)\longrightarrow (N,h)$ with $F_2(x_2)=F(x_1,x_2)$ for
fixed $x_1\in M$. Then, the bitension field of $F$ is given by
\begin{eqnarray}\label{BTP}
\tau_2(F)=&&\tau_2(F_1)\circ\pi_1+\tau^2(F_2)\circ\pi_2\\\notag
&&-J^{F_1\circ\pi_1}\big(\tau(F_2\circ\pi_2)\big)-J^{F_2\circ\pi_2}\big(\tau(F_1\circ\pi_1)\big).
\end{eqnarray}
In particular, for $F:(M_1\times M_2, g_1\times g_2)\longrightarrow
\r^n$, we have
\begin{eqnarray}\label{BTPR}
\tau_2(F)=&&(\Delta_{M_1}^2 F_1)\circ\pi_1+(\Delta_{M_2}^2F_2)\circ\pi_2+2\Delta_{M_1}\Delta_{M_2}\big(F_2\circ\pi_2\big),
\end{eqnarray}
where, $\Delta_{M_1}\Delta_{M_2}\big(F_2\circ\pi_2\big)=\Delta_{M_2}\Delta_{M_1}\big(F_1\circ\pi_1\big)$. Therefore, if one of $F_1$ and $F_2$ is proper
biharmonic and the other is harmonic, then $F$ is a proper biharmonic map.
\end{proposition}

\begin{proof}
Choose a local orthonormal frame $\{{e_{i}}\}_{i=1,\ldots,m}$ on
$M_1$ and a local orthonormal frame $\{{c_{a}}\}_{a=1,\ldots,n}$ on
$M_2$ so that $\{{e_{i}, c_a}\}_{i=1,\ldots,m,\; a=1,\ldots, n}$
form a local orthonormal frame on $M_1\times M_2$. Let
$\pi_k:M_1\times M_2\longrightarrow M_k, \pi_k(x_1,x_2)=x_k$, be the
projection onto the kth factor ($ k=1, 2$). Then,  it is easily seen
that
\begin{eqnarray}
{\rm d}F={\rm d}(F_1\circ\pi_1)+{\rm d}(F_2\circ\pi_2),\;\;
\end{eqnarray}

The tension field of $F$ is given (see e.g., (4.15) in \cite{EL1})
by
\begin{eqnarray}\notag
\tau(F) =&&\sum_{i=1}^{m}\{\nabla^{F_1\circ\pi_1}_{e_{i}}{\rm
d}(F_1\circ\pi_1)e_{i} -{\rm d}(F_1\circ\pi_1){\nabla^{M_1\times
M_2}}_{e_i}e_{i}\}\\\notag&&+\sum_{a=1}^{n}\{\nabla^{F_2\circ\pi_2}_{c_{a}}{\rm
d}(F_2\circ\pi_2)c_{a} -{\rm d}(F_2\circ\pi_2){\nabla^{M_1\times
M_2}}_{c_a}c_{a}\}
\\=&&\tau(F_1\circ\pi_1)+\tau(F_2\circ\pi_2).
\end{eqnarray}
Using the fact that the Jacobi operator is linear we have
\begin{eqnarray}\notag
\tau_2(F)=&&-J^F(\tau(F))=-J^F\big[\tau(F_1\circ\pi_1)+\tau(F_2\circ\pi_2)\big]\\\notag
&&=-J^F\big(\tau(F_1\circ\pi_1)\big)-J^F\big(\tau(F_2\circ\pi_2)\big)\\\notag
&&=-J^{F_1\circ\pi_1}\big(\tau(F_1\circ\pi_1)\big)-J^{F_2\circ\pi_2}\big(\tau(F_2\circ\pi_2)\big)\\\notag &&-J^{F_1\circ\pi_1}\big(\tau(F_2\circ\pi_2)\big)-J^{F_2\circ\pi_2}\big(\tau(F_1\circ\pi_1)\big)\\\notag
&&=\tau^2(F_1\circ\pi_1)+\tau^2(F_2\circ\pi_2)-J^{F_1\circ\pi_1}\big(\tau(F_2\circ\pi_2)\big)-J^{F_2\circ\pi_2}\big(\tau(F_1\circ\pi_1)\big)\\\label{24}
&&=\tau^2(F_1)\circ\pi_1+\tau^2(F_2)\circ\pi_2-J^{F_1\circ\pi_1}\big(\tau(F_2\circ\pi_2)\big)-J^{F_2\circ\pi_2}\big(\tau(F_1\circ\pi_1)\big),
\end{eqnarray}
where the last equality is obtained by using the fact that both
$\pi_1$ and $\pi_2$ are harmonic morphisms with constant dilations
$\lambda\equiv 1$ and hence biharmonic morphisms (\cite{Ou4}). The first statement of the
proposition follows from (\ref{24}).\\
\indent For the second statement, notice that when the target manifold of $F$ is a Euclidean space, we have
\begin{eqnarray}\notag\label{23}
\tau(F) =\Delta_{M_1}(F_1\circ\pi_1)+\Delta_{M_2}(F_2\circ\pi_2),
\end{eqnarray}
and a straightforward computation yields
\begin{eqnarray}\notag
\tau_2(F)=\Delta_{M_1}^2(F_1\circ\pi_1)+\Delta_{M_2}^2(F_2\circ\pi_2)+\Delta_{M_1}\Delta_{M_2}(F_2\circ\pi_2)+\Delta_{M_2}\Delta_{M_1}(F_2\circ\pi_2).
\end{eqnarray}
On the other hand, since $\pi_1: (M_1\times M_2, g_1\times g_2)\longrightarrow (M_1, g_1)$ is a Riemannian submersion with totally geodesic fibers, horizontal distribution $\mathcal{H}=TM_1$, and vertical distributuin $\mathcal{V}=TM_2$. Noting that $\Delta_{M_1}=\Delta_{\mathcal{H}}$ and $\Delta_{M_2}=\Delta_{\mathcal{V}}$, we can use a well-known fact (see \cite{BB}) that $\Delta_{\mathcal{H}}\circ\Delta_{\mathcal{V}}=\Delta_{\mathcal{V}}\circ\Delta_{\mathcal{H}}$ to have $\Delta_{M_1}\Delta_{M_2}\big(F_2\circ\pi_2\big)=\Delta_{M_2}\Delta_{M_1}\big(F_1\circ\pi_1\big)$.  Thus, we complete the proof of the proposition.
\end{proof}

\begin{remark}
It is very easy to see that the converse of Proposition \ref{ML} is
not true. For example, $F: S^1\times S^1\longrightarrow S^3$ defined
by \\$F(x,y)=(\frac{\sqrt{3}}{2}\, \cos x,\;\;
\frac{\sqrt{3}}{2}\,\sin x,\;\; \frac{1}{2} \,\cos y,\;\;
\frac{1}{2}\,\sin y)$ is harmonic (and hence biharmonic) but it is
not biharmonic with respect to either single variable.
\end{remark}

\begin{example}
The map $F:\r\times (\r^4\setminus \{0\})\longrightarrow \r^4$
$F(t,x)=\frac{tx}{|x|^2}$ with rational functions as component
functions is a proper biharmonic map. The effect of the map can be
interpreted as an inversion of the point $x\in \r^4$ about the unit
sphere $S^3\subset \r^4$ followed by a translation of $t$ unit along
the direction of $x$ for $t>0$ (and opposite direction of $x$ for
$t<0$). To see that this map is proper biharmonic we notice that for
each fixed $t_0$, the map $F(t_0,\cdot):\r^4\setminus
\{0\}\longrightarrow \r^4$ is a constant multiple of the inversion
about $3$-sphere which is  proper biharmonic by \cite{BK}. On the
other hand, for each fixed $x_0$, the map $F(\cdot,
x_0):\r\longrightarrow \r^4$ is a linear map, a straight line and
hence a geodesic, which is clearly harmonic. From these and
Proposition \ref{ML} we conclude that the map is indeed proper
biharmonic.
\end{example}

\begin{example}
The map $\phi: S^n(\frac{1}{\sqrt{2}})\times \r \longrightarrow
S^{n+1}\times \r$ with $\phi(x,t)=(x,\frac{1}{\sqrt{2}}, t)$ is a
proper biharmonic map. This  follows from Proposition \ref{ML} and
the fact that for each fixed $t$, the map is the inclusion $
S^n(\frac{1}{\sqrt{2}})\hookrightarrow S^{n+1}$ which is proper
biharmonic (\cite{CMO1}); and for each fixed $x\in
S^n(\frac{1}{\sqrt{2}})$, the map is a geodesic in $S^{n+1}\times
\r$ and hence harmonic.
\end{example}

{\bf 2.1  Biharmonic maps via orthogonal multiplications.} An
orthogonal multiplication is a bilinear map $f: \r^p\times
\r^q\longrightarrow \r^r$ such that $|f(x,y)|=|x||y|$. It is well
known (see, e.g., \cite{EL1}, Section (4.16)) that any orthogonal
multiplication restricts to a bi-eigenmap $f: S^{p-1}\times
S^{q-1}\longrightarrow S^{r-1}$ which is totally geodesic embedding
and hence a harmonic map in each variable separately. Furthermore,
for the orthogonal multiplications  $f: \r^n\times
\r^n\longrightarrow \r^n, n=1, 2, 4, 8$ defined by the usual
multiplications of algebras of real, complex, quaternionic, and
Cayley numbers, the Hopf construction maps $F: \r^n\times
\r^n\longrightarrow \r^{n+1}$ defined by
$F(x,y)=(2f(x,y),|x|^2-|y|^2)$ restrict to the Hopf fibrations
$S^{2n-1}\longrightarrow S^n$ which are harmonic maps. The following
theorem shows that any orthogonal multiplication can be used to
construct a proper biharmonic map from a torus into a sphere.
\begin{theorem}\label{OTh}
For any orthogonal multiplication $f: \r^p\times \r^q\longrightarrow
\r^n$, the map $\phi:  \r^{p}\times \r^{q}\longrightarrow \r^{n+1}$
defined  by
$\phi(x,y)=(\frac{1}{\sqrt{2}}f(x,y),\frac{1}{\sqrt{2}})$ restricts
to a proper biharmonic map $S^{p-1}\times S^{q-1}\longrightarrow
S^{n}$.
\end{theorem}
\begin{proof}
It is well known (see (4.16) in \cite{EL1} ) that an orthogonal multiplication $f: \r^p\times \r^q\longrightarrow
\r^n$ restricts to a harmonic map (using the same notation) $f:
S^{p-1}\times S^{q-1}\longrightarrow S^{n-1}$ with constant energy density. It is easily checked that  $\varphi:
S^{p-1}\times S^{q-1}\longrightarrow S^{n}(\frac{1}{\sqrt{2}})$ defined by $\varphi(x,y)=\frac{1}{\sqrt{2}}f(x,y)$  is again a harmonic map with constant energy density. On the other hand, notice that, the map $\phi:  S^{p-1}\times S^{q-1}\longrightarrow S^{n}$ defined  by the restriction of  $\phi:  \r^{p}\times \r^{q}\longrightarrow \r^{n+1}$,  $\phi(x,y)=(\frac{1}{\sqrt{2}}f(x,y),\frac{1}{\sqrt{2}})$,  is the composition
of $\varphi: S^{p-1}\times S^{q-1}\longrightarrow S^{n-1}(\frac{1}{\sqrt{2}})$,
$\varphi(x,y)=\frac{1}{\sqrt{2}}f(x, y)$, followed by the inclusion map
${\bf i}:S^{n-1}(\frac{1}{\sqrt{2}})\hookrightarrow S^n$.  It follows from Theorem 1.1 in \cite{LOn1} that the
map $\phi: S^{p-1}\times S^{q-1}\longrightarrow
S^{n-1}(\frac{1}{\sqrt{2}})\hookrightarrow S^n$, $\phi(x,y)=(\frac{1}{\sqrt{2}}f(x,y),\frac{1}{\sqrt{2}})$,  is a proper
biharmonic map. Thus, we obtain the theorem.
\end{proof}
\begin{corollary}\label{TS}{\bf(Biharmonic tori in $S^n$)}
For any $k, n \geq 2$, there is a proper biharmonic map
$T^k=S^1\times\ldots \times S^1\longrightarrow S^n$ from flat torus
into $n$-sphere. In particular, the map
$\phi(t,s)=(\frac{1}{\sqrt{2}}\cos (t+s), \frac{1}{\sqrt{2}}\sin
(t+s), \frac{1}{\sqrt{2}})$ is a proper biharmonic map from
$2$-torus $T^2$ into $2$-sphere $S^2$.
\end{corollary}

\begin{proof}
The proper biharmonic flat torus $\phi: S^1 \times
S^1\longrightarrow S^2$ is obtained by applying Theorem \ref{OTh}
with the orthogonal multiplication $f:\mathbb{C}\times
\mathbb{C}\longrightarrow \mathbb{C}$ defined by the product of
complex numbers, i.e., $f(z,w)=zw$. In fact, let $(z, w)=(e^{it},
e^{is})\in S^1\times S^1$. then, By Theorem \ref{OTh}, the map can
be expressed as $\phi(t,s)=(\frac{1}{\sqrt{2}}\cos (t+s),
\frac{1}{\sqrt{2}}\sin (t+s), \frac{1}{\sqrt{2}})$. Note that with
respect to each variable  the map $\phi$ is a proper biharmonic
curve (up to an affine transformation of arc-length parameter) in
$S^2$ (\cite{CMO1}). By totally geodesically embedding $S^2$ into
$S^n$ and a result in \cite{Ou2} stating that totally geodesically
immersing the target manifold of a biharmonic map into another
manifold does not change the biharmonicity of the map we obtain the
flat torus in $S^n$.
\end{proof}
\begin{remark}
Note that the proper biharmonic map $T^2\longrightarrow S^2$
constructed in Corollary \ref{TS} is not onto so the degree of the
map is $0$. It would be interesting to know if there exists proper
biharmonic map of degree $\pm 1$ from $T^2$ to $S^2$ as it was
showed by J. Eells and J.C. Wood \cite{EW} that there exists no
harmonic map from $T^2$ to $S^2$ (whatever the metrics chosen) in
the homotopy class of Brower degree $\pm 1$.
\end{remark}

\begin{corollary}\label{TS4}
$(${\bf A flat torus in $S^4$}$)$ The map $\phi: \phi: S^1 \times
S^1\longrightarrow S^4$ defined by
$$\phi(t,s)=\frac{1}{\sqrt{2}}(\cos t \cos s,\;\; \cos
t\sin s,\;\; \sin t \cos s, \;\;\sin t\sin s,\;\; 1)$$ is a proper
biharmonic map.
\end{corollary}
\begin{proof}
This is obtained by applying Theorem \ref{OTh} with the orthogonal
multiplication $f:\mathbb{R}^2\times \mathbb{R}^2\longrightarrow
\mathbb{R}^4$ defined (c.f., \cite{Pa}) by $f(x,y)=(x_1y_1,x_2y_1,
x_1y_2, x_2y_2)$. Note that in this case, there is an interesting
way to see that the map is a proper biharmonic curve in $S^4$ with
respect to each variable separately. For instance, with respect to
$t$-variable ($s$ is fixed), the map becomes a curve
$\phi(t,s)=c_1\cos t+c_2 \sin t +c_3$, where
\begin{eqnarray}\notag
c_1&=&\frac{1}{\sqrt{2}}( \cos s,\sin s, 0,0,0),\;\;
c_2=\frac{1}{\sqrt{2}}(0,0,\cos s,\sin s,0),\\\notag
c_3&=&(0,0,0,0,\frac{1}{\sqrt{2}})
\end{eqnarray}
are three mutually orthogonal vectors in $\mathbb{R}^5$. It follows
from Proposition 4.4 in \cite{CMO2} that the $t$-curve is a proper
biharmonic curve in $S^4$ (up to an affine transformation of
arc-length parameter). Similarly, with respect to $s$-variable, the
map $\phi$ is also a proper biharmonic curve.
\end{proof}
\begin{remark}
(1) A similar construction applies to the family of the orthogonal
multiplications $f:\mathbb{R}^2\times \mathbb{R}^2\longrightarrow
\mathbb{R}^4$ defined in \cite{Pa} will produce a family of proper
flat tori in $S^4$.\\
(2) Note that none of the proper biharmonic maps constructed above
is an isometric immersion. So they are different from the closed
orientable proper biharmonic embedded surfaces of any genus in $S^4$
described in \cite{CMO2}.
\end{remark}

The following proposition shows that orthogonal multiplications can
also be used to construct proper biharmonic maps into Euclidean
space.
\begin{proposition}\label{euclid}
Let $f: \r^p\times \r^q\longrightarrow \r^n$ be an orthogonal
multiplication. If one of $\varphi:M\longrightarrow \r^{p}$ and
$\psi:N\longrightarrow \r^{q}$ is proper biharmonic and the other is harmonic, then the map $f\circ
(\varphi,\psi):M\times N\longrightarrow \r^{n}$ is proper
biharmonic.
\end{proposition}

\begin{proof}
This follows from Proposition 2.1 and the fact that with respect to each
variable, the orthogonal multiplication $f: \r^p\times
\r^q\longrightarrow \r^n$ is a homothetic totally geodesic
imbedding.
\end{proof}

\begin{example}
Let $f: \r^4\times \r^4\longrightarrow \r^4$, $f(x,y)=xy$ be the
orthogonal multiplication defined by the product of quaternions.
 Let $\varphi:\r^4\setminus \{0\}\longrightarrow \r^4$ with $\varphi(x)=\frac{x}{|x|^2}$ be the
 inversion about $3$-sphere, which is proper biharmonic. Then, a straightforward computation of the bi-Laplacian of component functions shows that the map 
 $f\circ(\varphi,\varphi): (\r^4\setminus \{0\})\times (\r^4\setminus \{0\}) \longrightarrow
 \r^4$ defined by $f\circ(\varphi,\varphi)(x,y)=\frac{xy}{|x|^2|y|^2}$ is not a biharmonic map. However, the map $f\circ(\varphi,{\rm Id}): (\r^4\setminus \{0\})\times (\r^4\setminus \{0\}) \longrightarrow
 \r^4$ defined by $f\circ(\varphi, {\rm Id})(x,y)=\frac{xy}{|x|^2}$ is a proper biharmonic map.\\
\end{example}

{\bf 2.2 More on biharmonic maps between spheres and into spheres}

There are many harmonic maps between spheres which include the
standard minimal isometric immersions (embeddings) of spheres into
spheres. To the author's knowledge, the only known examples of
proper biharmonic maps between Euclidean spheres seem to be the
following:
\begin{itemize}
\item[(1)] the inclusion map (\cite{CMO1}) ${\bf i}:
S^m(\frac{1}{\sqrt{2}})\hookrightarrow S^{m+1}, \;{\bf
i}(x)=(x,\frac{1}{\sqrt{2}})$ for $x\in \r^{m+1}$ with
$|x|=\frac{1}{\sqrt{2}}$, or more generally, homothetic immersion
$\phi: S^m \hookrightarrow S^{m+1}, \;{\bf
i}(x)=(\frac{x}{\sqrt{2}},\frac{1}{\sqrt{2}})$ for $x\in \r^{m+1}$
with $|x|=1$ (see \cite{Ou2} for details); \item[(2)] The
composition of minimal isometric immersion
$S^k(\frac{1}{\sqrt{2}})\hookrightarrow S^m(\frac{1}{\sqrt{2}})$
followed by the inclusion $S^m(\frac{1}{\sqrt{2}})\hookrightarrow
S^{m+1}$ (\cite{CMO2}, \cite{LOn2}), the composition of Veronese
map: $S^2(\sqrt{3})\hookrightarrow
S^4(\frac{1}{\sqrt{2}})\hookrightarrow S^5$ (\cite{CMO2});
\item[(3)] The composition of the Hopf map and the inclusion:
$S^3(\frac{1}{\sqrt{2}})\longrightarrow
S^2(\frac{1}{\sqrt{2}})\hookrightarrow S^3$ (\cite{LOn1}), and more
generally, the pull-backs of proper biharmonic maps
$S^m(\frac{1}{\sqrt{2}})\hookrightarrow S^{m+1} (m=2, 4, 8)$ by the
Hopf fibrations $S^{2m-1}\longrightarrow S^m(\frac{1}{\sqrt{2}})$
(\cite{Ou2}).
\end{itemize}

The following construction can be used to produce many proper
biharmonic maps between Euclidean spheres.

\begin{theorem}\label{S}
For any eigenmap $f:S^{m-1}\longrightarrow S^{n-1}$ there is an
associated proper biharmonic map $F:S^{m-1}\longrightarrow S^{n}$
defined by $F(x)=(\frac{1}{\sqrt{2}}f(x),\frac{1}{\sqrt{2}})$.
\end{theorem}

\begin{remark}
As we can see from Theorem \ref{S} and the introduction in this
section that all known examples of proper biharmonic maps between
spheres come from the construction given in Theorem \ref{S}. It
would be interesting to know whether there is other type of proper
biharmonic maps between spheres.
\end{remark}

\begin{theorem}\label{JI}
Let $\varphi: M\longrightarrow S^p$ and $\psi: N\longrightarrow S^q$
be harmonic maps with constant energy density from compact
manifolds, and ${\bf j}:S^p(\frac{1}{\sqrt{2}})\times
S^q(\frac{1}{\sqrt{2}})\hookrightarrow S^{p+q+1}, p\ne q, $ be the
standard inclusion. Then, $\phi={\bf
j}\circ(\frac{\varphi}{\sqrt{2}},\frac{\psi}{\sqrt{2}}):M\times
N\longrightarrow S^{p+q+1}$ is a proper biharmonic map.
\end{theorem}

\begin{corollary}
$(1)$ Let $\varphi: S^m\longrightarrow S^p$ and $\psi:
S^n\longrightarrow S^q$ be any eigenmaps, then the map
$\phi:S^m\times S^n\longrightarrow S^{p+q+1}$ with
$\phi(x,y)=\frac{1}{\sqrt{2}}(\varphi(x), \psi(y))$ is a proper
biharmonic map for $p\ne q$;\\
$(2)$ Let $\varphi: S^{m_1}\times S^{m_2}\longrightarrow S^p$ and
$\psi: S^{n_1}\times S^{n_2}\longrightarrow S^q$ be any bi-eigenmaps
defined by orthogonal multiplications, then the map
$\phi:S^{m_1}\times S^{m_2}\times S^{n_1}\times
S^{n_2}\longrightarrow S^{p+q+1}$ with
$\phi(x_1,x_2,y_1,y_2)=\frac{1}{\sqrt{2}}(\varphi(x_1,
x_2),\;\;\psi(y_1, y_2))$ is a proper biharmonic map for $p\ne q$;
\end{corollary}

\begin{remark}
When $\varphi: S^k(\frac{1}{\sqrt{2}})\longrightarrow S^p,\;
\varphi(x)=\sqrt{2}\,x$ and $\psi:
S^k(\frac{1}{\sqrt{2}})\longrightarrow S^q,\; \psi(y)=\sqrt{2}y$ be
the homothetic minimal embedding of spheres into spheres, our
Theorem \ref{JI} recover part of results in Proposition 3.10 and
Theorem 3.11 in \cite{CMO2}
\end{remark}
\begin{example}
Let $\varphi: S^1\longrightarrow S^1$, $\varphi(z)=z, z\in
\mathbb{C}$ be the identity map and $\psi: S^3\longrightarrow S^2$
be the Hopf fibration defined by $\psi(x,y)=(2xy, |x|^2-|y|^2), x,
y\in \mathbb{C}$. Then, by Theorem \ref{JI}, we have a proper
biharmonic map $\phi:S^1\times S^3\longrightarrow S^{4}$ with
$\phi(z, x,y)=\frac{1}{\sqrt{2}}(z, 2xy,|x|^2-|y|^2)$.
\end{example}

\begin{example}
Let $\varphi: S^1\longrightarrow S^1$, $\varphi(z)=z, z\in
\mathbb{C}$ be the identity map and $\psi: S^2\longrightarrow S^2$
be a family of isometries (with parameter $t$) defined by
$\psi(y_1,y_2,y_3)=(y_1\cos\, t+y_2\sin\, t,\;\; -y_1\sin\,
t+y_2\cos\, t,\;\; y_3), y\in S^2$. Then, by Theorem \ref{JI}, we
have a family of proper biharmonic map $\phi:S^1\times
S^2\longrightarrow S^{4}$ with $\phi_t(z,y)=\frac{1}{\sqrt{2}}(z,
y_1\cos\, t+y_2\sin\, t,\;\; -y_1\sin\, t+y_2\cos\, t,\;\; y_3)$. In
particular, when $t=0$, the proper biahrmonic map becomes $\phi_0(z,
y)=\frac{1}{\sqrt{2}}(z,\;\; y)$, the standard homothetic embedding
which is the composition $ S^1\times S^2\hookrightarrow
S^1(\frac{1}{\sqrt{2}})\times S^2(\frac{1}{\sqrt{2}})\hookrightarrow
S^4$.
\end{example}

For classification results on proper biharmonic isometric immersions
into spheres, i.e., proper biharmonic submanifolds in spheres see a
recent survey in \cite {BMO}.

{\bf 2.3 Biharmonic maps by complete lifts.} Let $\varphi:
\r^m\supseteq U\longrightarrow \r^n$ be a map into Euclidean space.
The complete lift of  $\varphi$ is defined  in \cite{Ou1} to be a
map $\phi: U \times \r^m\longrightarrow \r^n$ with $\phi(x,y)=({\rm
d}_x\varphi)(y)$. It was proved in \cite{Ou1} that the complete lift
of a harmonic map is a harmonic map, the complete lift of a
quadratic harmonic morphism is again a quadratic harmonic morphism.
We will show that the complete lift method can also be used to
produce new biharmonic maps from given ones.
\begin{proposition}\label{CPL}
The complete lift of a proper biharmonic map is a proper biharmonic
map.
\end{proposition}
\begin{proof}
Let $\varphi(x)=(\varphi^1(x), \varphi^2(x),\ldots, \varphi^n(x))$,
then, by definition,\\
$\phi(x,y)=(\sum_{i=1}^m\frac{\partial\varphi^1}{\partial x_i}y_i,
\sum_{i=1}^m\frac{\partial\varphi^2}{\partial x_i}y_i,\ldots,
\sum_{i=1}^m\frac{\partial\varphi^n}{\partial x_i}y_i)$. It is easy
to check that a map into Euclidean space is biharmonic if and only
if each of its component function is biharmonic. Let
$\Delta_{(x,y)}$ denote the Laplacian on $U \times \r^m$, clearly,
$\Delta_{(x,y)}=\Delta_{x}+\Delta_{y}$. Therefore,
\begin{eqnarray}
\Delta^2_{(x,y)}\phi^a(x,y)&=&\Delta_{(x,y)}\big(\Delta_{(x,y)}(\sum_{i=1}^m\frac{\partial\varphi^a}{\partial
x_i}y_i)\big)\\\notag
&=&\sum_{i=1}^m\Delta^2_{x}(\frac{\partial\varphi^a}{\partial
x_i})y_i=\sum_{i=1}^m\frac{\partial}{\partial
x_i}\big(\Delta^2_{x}\varphi^a\big)y_i,
\end{eqnarray}
from which we obtain the proposition.
\end{proof}

\begin{example}
Let $\varphi:\r^4\setminus\{0\}\longrightarrow \r^4$ be the
inversion about 3-sphere in $\r^4$ defined by $\varphi(x)=x/|x|^2$
which is known (\cite{BK}) to be a proper biharmonic map. Its
complete lift $\phi :\r^4\setminus\{0\}\times \r^4\longrightarrow
\r^4$ is defined by
\begin{eqnarray}\notag
\phi (x,y)=|x|^{-4}\big((|x|^2-2x_1^2)y_1-2x_1\sum_{i\neq
1}^4x_iy_i,\;\;(|x|^2-2x_2^2)y_2-2x_2\sum_{i\neq 2}^4x_iy_i,\\
 (|x|^2-2x_3^2)y_3-2x_3\sum_{i\neq 3}^4x_iy_i,
\;\;(|x|^2-x_4^2)y_4-2x_4\sum_{i\neq 4}^4x_iy_i\big),
\end{eqnarray}
which is a proper biharmonic map by Proposition \ref{CPL}. Note that
this is another example of proper biharmonic map defined by rational
functions.
\end{example}

{\bf 2.4 Biharmonic maps by direct sum construction.} Let  $
\varphi:M \longrightarrow {\Bbb R}^{n}$ and $\psi:N \longrightarrow
{\Bbb R}^{n}$ be two  maps. Then the {\bf direct sum} of $ \varphi$
 and $\psi$ is defined (see \cite{Ou0} and also \cite{BW1}) to be the map
\begin{equation}
\varphi \oplus \psi: M \times N \longrightarrow {\Bbb R}^{n} \notag
\end{equation}
given by
\begin{equation}
(\varphi \oplus \psi)(p,q) = \varphi(p) + \psi(q)\notag
\end{equation}
where $ M \times N$ is the product of $M$ and $N$, endowed with  the
Riemannian product metric $G = g\times h$. An immediate consequence
of Proposition \ref{ML} is the following
\begin{corollary}\label{C3}
The direct sum of biharmonic maps is again a biharmonic map; the
direct sum of a harmonic map and a proper biharmonic map is a proper
biharmonic map.
\end{corollary}
\begin{example}
The map $\phi:(\r^3,g=e^{x_2}(dx_1^2+dx_2^2)+dx_3^2)\longrightarrow
\r^3$ defined by $\phi(x_1,x_2, x_3)=(\cos\,x_1+3x_3,\;
\sin\,x_1+2x_3,\; x_2-x_3)$ is a proper biharmonic map. This follows
from Corollary \ref{C3} and the fact that the map $\phi$ is the
direct sum of the the geodesic (map) $\psi:\r^1\longrightarrow \r^3$
defined by $\psi(x_3)=(3x_3,2x_3,-x_3)$ and the map
$\varphi:(\r^2,g=e^{x_2}(dx_1^2+dx_2^2))\longrightarrow \r^3$,
 $\varphi(x_1,x_2)=(\cos\,x_1,
\sin\,x_1, x_2)$, which is a proper biharmonic conformal immersion
of  $\r^2$ into Euclidean space $\r^3$ (\cite{Ou3}).
\end{example}

\section{Biharmonic maps into a product manifold}

The graph of a map $\psi: (M,g)\longrightarrow (N, h)$ is defined to
be the map $\phi: M\longrightarrow (M\times N, g\times h)$ with
$\phi(x)=(x, \psi(x))$ which is easily checked to be an embedding.
So with respect to the pull-back metric $\phi^*(g+h)=g+\psi^*h$ the
graph $\phi$ is an isometric embedding whilst with respect to the
original metric $g$ on $M$ it need not be so. The harmonicity of the
graph of a  map $\psi$ with respect to the pull-back metric
$\phi^*(g+h)=g+\psi^*h$ and the original metric $g$ on $M$ were
studied in \cite{ES} and \cite{Ee}.
\begin{proposition}$($\cite{ES}, 2 (E)$)$ \label{HP1}
Let $\psi: (M,g)\longrightarrow (N, h)$ be a map. Then, the graph
$\phi: (M,g)\longrightarrow (M\times N, g\times h)$ with
$\phi(x)=(x, \psi(x))$ is a harmonic map if and only if the map
$\psi: (M,g)\longrightarrow (N, h)$ is a harmonic map.
\end{proposition}
\begin{proposition} $($\cite{Ee}$)$
Let $\psi: (M,g)\longrightarrow (N, h)$ be a map. Then, the graph
$\phi: (M,\phi^*(g\times h))\longrightarrow (M\times N, g\times h)$
with $\phi(x)=(x, \psi(x))$ is a harmonic map if and only if  both
$\pi_1\circ \phi$ and $\pi_2\circ \phi=\psi$ are harmonic, where
$\pi_1:(M\times N, g\times h)\longrightarrow (M,g)$ and
$\pi_2:(M\times N, g\times h)\longrightarrow (N,h)$ are the
projections onto the first and the second factor respectively.
\end{proposition}

We will show that these results generalize to the case of biharmonic
maps. Actually, we will prove that the generalizations follow from
the following theorem.

\begin{theorem}\label{BHP1}{\bf(Biharmonic maps into product
spaces)} Let $\varphi: (M,g)\longrightarrow (N_1, h_1)$ and $\psi:
(M,g)\longrightarrow (N_2, h_2)$ be two maps. Then, the map $\phi:
(M,g)\longrightarrow (N_1\times N_2, h_1\times h_2)$ with
$\phi(x)=(\varphi(x), \psi(x))$ is biharmonic  if and only if both
map $\varphi$ and $\psi$ are biharmonic. Furthermore, if one of
$\varphi$ and $\psi$ is harmonic and the other is a proper
biharmonic map, then $\phi$ is a proper biharmonic map.
\end{theorem}
\begin{proof}
It is easily seen that
\begin{eqnarray}\label{gd5}
{\rm d}\phi(X)= {\rm d}\varphi(X)+{\rm d}\psi(X), \;\; \forall\;
X\in TM.
\end{eqnarray}
 It follows that
\begin{eqnarray}\label{gd6}
\nabla^{\phi}_{X}{\rm d}\phi(Y)=\nabla^{\phi}_{X}{\rm
d}\varphi(Y)+\nabla^{\phi}_{X}{\rm d}\psi(Y), \;\; \forall\; X,\;
Y\in TM.\\\notag
\end{eqnarray}
Let $\{e_1, \ldots, e_m\}$ be a local orthonormal frame on $(M,g)$
and let $Y=Y^ie_i$, then $ {\rm
d}\varphi(Y)=Y^i\varphi^a_i(E_a\circ\varphi)$ for some function
$\varphi^a_i$ defined locally on $M$. A straightforward computation
yields
\begin{eqnarray}\notag
\nabla^{\phi}_{X}{\rm
d}\varphi(Y)&=&\nabla^{\phi}_{X}(Y^i\varphi^a_i)(E_a\circ\varphi)\\\label{gd7}
&=&[X(Y^i\varphi^a_i)](E_a\circ\varphi)+(Y^i\varphi^a_i)\nabla^{N_1}_{{\rm
d}\varphi(X)}(E_a\circ\varphi)\\\notag
&=&\nabla^{\varphi}_{X}(Y^i\varphi^a_i)(E_a\circ\varphi)=\nabla^{\varphi}_{X}{\rm
d}\varphi(Y).
\end{eqnarray}
Similarly, we have
\begin{eqnarray}\label{gd8}
\nabla^{\phi}_{X}{\rm d}\psi(Y)=\nabla^{\psi}_{X}{\rm d}\psi(Y).
\end{eqnarray}

By using (\ref{gd5}),  (\ref{gd6}), (\ref{gd7}), and (\ref{gd8}) we
have
\begin{eqnarray}\notag
\tau(\phi)&=&\sum_{i=1}^m\{\nabla^{\phi}_{e_i}{\rm d}\phi(e_i)-{\rm
d}\phi(\nabla^M_{e_i}e_i)\}\\\label{Pt}
&=&\sum_{i=1}^m\{\nabla^{\varphi}_{e_i}{\rm d}\varphi(e_i)-{\rm
d}\varphi(\nabla^M_{e_i}e_i)+\nabla^{\psi}_{e_i}{\rm
d}\psi(e_i)-{\rm d}\psi(\nabla^M_{e_i}e_i)\}\\\notag &=&
\tau(\varphi)+\tau(\psi).
\end{eqnarray}

To compute the bitension field of the map $\phi$, we notice that
$\tau(\varphi)$ is tangent to $N_1$ whilst $\tau(\psi)$ is tangent
to $N_2$.  We use the property of the curvature of the product
manifold
 to have
\begin{eqnarray}\notag\label{Pcur}
&&{\rm R}^{N_1\times N_2}({\rm d}\phi(e_i), \tau(\phi)){\rm
d}\phi(e_i)\\\notag &=& {\rm R}^{N_1}({\rm d}\varphi(e_i),
\tau(\varphi)){\rm d}\varphi(e_i)+{\rm R}^{N_2}({\rm d}({\rm
d}\psi(e_i), \tau(\psi)){\rm d}\psi(e_i).
\end{eqnarray}
Therefore,
\begin{eqnarray}\notag\label{Btpro}
\tau^2(\phi)&=&\sum_{i=1}^m\{\nabla^{\phi}_{e_i}\nabla^{\phi}_{e_i}\tau
(\phi)-\nabla^{\phi}_{\nabla^M_{e_i}e_i}\tau(\phi)-{\rm
R}^{N_1\times
N_2}({\rm d}\phi(e_i), \tau(\varphi)){\rm d}\phi(e_i)\}\\
&=&\sum_{i=1}^m\{\nabla^{\varphi}_{e_i}\nabla^{\varphi}_{e_i}\tau
(\varphi)-\nabla^{\varphi}_{\nabla^M_{e_i}e_i}\tau(\varphi) -{\rm
R}^{N_1}({\rm d}\varphi(e_i), \tau(\varphi)){\rm
d}\varphi(e_i)\}\\\notag
&&+\sum_{i=1}^m\{\nabla^{\psi}_{e_i}\nabla^{\psi}_{e_i}\tau
(\psi)-\nabla^{\psi}_{\nabla^M_{e_i}e_i}\tau(\psi) -{\rm
R}^{N_2}({\rm d}\psi(e_i), \tau(\psi)){\rm d}\psi(e_i)\}\\\notag
&=&\tau^2(\varphi)+\tau^2(\psi),
\end{eqnarray}
from which, together with (\ref{Pt}) the theorem follows.
\end{proof}

The following corollary generalizes Proposition \ref{HP1} in
\cite{ES} and can be used to produce proper biharmonic maps from
given ones.
\begin{corollary}
Let $\psi: (M,g)\longrightarrow (N, h)$ be a map. Then, the graph
$\phi: (M,g)\longrightarrow (M\times N, g\times h)$ with
$\phi(x)=(x, \psi(x))$ is a biharmonic map if and only if the map
$\psi: (M,g)\longrightarrow (N, h)$ is a biharmonic map.
Furthermore, if $\psi$ is proper biharmonic, then so is the graph.
\end{corollary}
\begin{proof}
The corollary follows from Theorem \ref{BHP1} with
$\varphi:(M,g)\longrightarrow (M,g)$ being identity map which is
harmonic.
\end{proof}
\begin{example}
The map $\phi: \r^4\setminus\{0\}\longrightarrow \r^4\times \r^4$
given by $\phi(x)=(x, x/|x|^2)$ is a proper biharmonic map. This
follows from Theorem \ref{BHP1} and the fact that $\phi$ is the
graph of the inversion $\psi:\r^4\setminus\{0\}\longrightarrow \r^4$
defined by $\psi(x)=x/|x|^2$ which is known (\cite{BK}) to be a
proper biharmonic map.
\end{example}
\begin{example}
The map $\phi:(\r^2,g=e^{x_2}(dx_1^2+dx_2^2))\longrightarrow
(\r^2\times \r^3,e^{x_2}(dx_1^2+dx_2^2)+dx_3^2+dx_4^2+dx_5^2)$
defined by $\psi(x_1,x_2)=(x_1, x_2, \cos\,x_1, \sin\,x_1, x_2)$ is
a proper biharmonic map. This is because the map $\phi$ is the graph
of the map $\psi:(\r^2,g=e^{x_2}(dx_1^2+dx_2^2))\longrightarrow
\r^3$, $\psi(x_1,x_2)=(\cos\,x_1, \sin\,x_1, x_2)$, which is a
proper biharmonic conformal immersion of  $\r^2$ into Euclidean
space $\r^3$ (\cite{Ou3}).
\end{example}

Another consequence of Theorem \ref{BHP1} is the following

\begin{corollary}\label{PBH2}
Let $\psi: (M,g)\longrightarrow (N, h)$ be a map. Then,\\ $(1)$ the
graph $\phi: (M,\phi^*(g\times h))\longrightarrow (M\times N,
g\times h)$ with $\phi(x)=(x, \psi(x))$, is a biharmonic isometric
embedding if and only if both maps $\varphi:
(M,g+\psi^*h)\longrightarrow (M,g)$ with $\varphi(x)=x$ and $\psi:
(M,g+\psi^*h)\longrightarrow (N, h)$ are biharmonic. Furthermore, if
one of the component maps is proper biharmonic then so is the graph
$\phi$;\\$(2)$ A submanifold $\psi: (M,g)\longrightarrow (N, h)$ is
biharmonic if and only if its graph $\phi: (M,g)\longrightarrow
(M\times N, g\times h)$ with $\phi(x)=(x, \psi(x))$ is a biharmonic
 embedded submanifold.\\
 $(3)$ The graph of $\psi: (M^2,g)\longrightarrow (N, h)$ is
 a biharmonic isometric embedding  if and only if the map $\psi: (M^2,g+\psi^*h)\longrightarrow (N,
 h)$ is biharmonic.
\end{corollary}
\begin{proof}
The Statement (1) follows from Theorem \ref{BHP1} immediately. For
Statement (2), note that if $\psi: (M,g)\longrightarrow (N, h)$ is
an isometric immersion, then $\phi^*(g\times h)=2g$, from which and
Theorem \ref{BHP1} we obtain the required result. To prove Statement
(3), first note that, by Theorem \ref{BHP1}, the biharmonicity of
the graph $\phi$ is equivalent to the biharmonicity of $\varphi:
(M^2,g+\psi^*h)\longrightarrow (M^2,
 g)$ and $\psi: (M^2,g+\psi^*h)\longrightarrow (N,
 h)$. Since $\varphi$ is an identity map between two Riemann
 surfaces, one can easily check that $\varphi$ is harmonic. It
 follows that $\phi$ is biharmonic if and only if $\psi$ is
 biharmonic.
\end{proof}

\section{Biharmonic graphs and Chen's conjecture}
Concerning biharmonic submanifolds of Euclidean space  we have the
following\\
{\bf Chen's Conjecture:} any biharmonic submanifold of Euclidean
space is minimal. \\

Jiang \cite{Ji2}, Chen-Ishikawa \cite{CI} proved that any biharmonic
surface in $\r^{3}$ is minimal; Dimitri$\acute{\rm c}$ \cite{Di}
showed that any biharmonic curves in $\r^n$ is a part of a straight
line, any biharmonic submanifold of finite type in $\r^n$ is
minimal, any pseudo-umbilical submanifolds $M^m\subset \r^n$ with
$m\ne 4$ is minimal, and any biharmonic hypersurface in $\r^n$ with
at most two distinct principal curvatures is minimal; it is proved
in \cite{HV} that any biharmonic hypersurface in $\r^4$ is minimal.
However, the conjecture is still open.\\

As any hypersurface of Euclidean space is locally the graph of a
real-valued function, we believe that the next strategic step in
attacking Chen's conjecture is to check whether a biharmonic graph
in Euclidean space is minimal. The following theorem gives
conditions on a real-valued function whose graphs produce biharmonic
hypersurfaces in a Euclidean space, from which we obtain an
equivalent analytic formulation of Chen's conjecture.
\begin{theorem}\label{Conj}
(I) A function $f:\r^m\supseteq D\longrightarrow \r$ has biharmonic
graph \\$\{(x, f(x)): x\in D\} \subseteq \r^{m+1}$ if and only if of
$f$ is a solution of
\begin{equation}\label{Bg}
\begin{cases}
\Delta ^2 f=0,\\(\Delta  f_k)\Delta f+2g(\nabla f_k,\nabla \Delta
f)=0, \;\; k=1,2,\ldots, m,
\end{cases}
\end{equation}
where the Laplacian $\Delta$ and the gradient $\nabla$ are taken
with respect to the induced metric $g_{ij}=\delta_{ij}+f_if_j$.\\
(II) The following statements are equivalent:\\
(i) Any biharmonic hypersurface in Euclidean space $\r^{m+1}$ is
minimal (Chen's Conjecture for biharmonic hypersurfaces \cite{CH});\\
(ii) Any solution of Equation (\ref{Bg}) is a harmonic function,
i.e., $\Delta f=0$.
\end{theorem}

\begin{proof}
To prove Statement (I), let $f:\r^m\supseteq D\longrightarrow \r$
with $f=f(x_1,\ldots, x_m)$ be a function. Then, by Corollary
\ref{PBH2}, the graph $\phi:\r^m\supseteq D\longrightarrow \r^{m+1}$
with $\phi(x)=(x_1,\ldots, x_m, f(x_1,\ldots, x_m))$ is a biharmonic
hypersurface if and only if all component functions are biharmonic
with respect to the induced metric
$g=(g_{ij})=(\delta_{ij}+f_if_j)$. This is equivalent to $f$ being a
solution of
\begin{equation}\label{Bggd}
\begin{cases}
\Delta ^2 f=0,\\\Delta^2 x_k=0, \;\;\;k=1,2,\ldots, m.
\end{cases}
\end{equation}

To prove the equivalence of (\ref{Bg}) and (\ref{Bggd}) we need to
find  a way to compute the Laplacian on the hypersurface. To this
end, we use the standard Cartesian coordinates $( x_1,\ldots,
x_m,x_0)$  in $\r^{m+1}$ and with respect to which the standard
Euclidean metric is given by $h_0=dx_1^2+\ldots + dx_m^2+dx_0^2$. We
will use the notations $\partial_i=\frac{\partial}{\partial x_i},
i=0, 1, 2, \ldots, m$,  $f_i=\frac{\partial f}{\partial x_i}$, and
$f_{ij}=\frac{\partial^2 f}{\partial x_j\partial x_i}$.\\

{\bf Claim:} The Laplacian on the hypersurface with respect to the
induced metric $g$ is given by
\begin{eqnarray}\label{Lap}
\Delta u=g^{ij}u_{ij}-\frac{g^{ij}f_{ij}}{1+|\nabla_0 f|^2}\sum_{i=1}^m f_iu_i
\end{eqnarray}
where  $u$ is a function defined on $D\subseteq\r^m$. In particular, 
\begin{equation}\label{Laf}
 \Delta f=\frac{g^{ij}f_{ij}}{1+|\nabla_0 f|^2}.
\end{equation}

{\em Proof of the Claim:} It is easy to check that
\begin{equation}\label{Frame}
\begin{cases}
X_i={\rm d}\phi(\partial_i)=\partial_i+f_i\partial_0, \;\;i=1,
2,\ldots, m,\\
\xi=(-\nabla_0f+\partial_0)/\sqrt{1+|\nabla_0 f|^2\;},
\end{cases}
\end{equation}
where $\nabla_0$ denoting the Euclidean gradient, form a frame (not
necessarily an orthonormal one) on $\r^{m+1}$ adapted to the
hypersurface with $\xi$ being the unit normal vector field. Clearly,
the induced metric on the hypersurface has components
\begin{eqnarray}
(g_{ij})=(\langle X_i,X_j\rangle)=(\delta_{ij}+f_if_j),\;\;\;
(g^{ij})=(g_{ij})^{-1}.
\end{eqnarray}

Let ${\bar\nabla}$ denote the connection of the ambient Euclidean
space. Then,  we have
\begin{eqnarray}
&&{\bar\nabla}_{X_i}X_j={\bar\nabla}_{(f_i\partial_0+\partial_i)}(f_j\partial_0+\partial_j)=f_{ij}\partial_0.
\end{eqnarray}

The second fundamental form of the hypersurface with respect to the
frame $\{X_i\}$ is give by
\begin{eqnarray}
&&b_{ij}=b(X_i,X_j)=\langle{\bar\nabla}_{X_i}X_j,
\xi\rangle=\frac{f_{ij}}{\sqrt{1+|\nabla_0 f|^2\;}}.
\end{eqnarray}

We can apply the Gauss formula
${\bar\nabla}_{X_i}X_j-\nabla_{X_i}X_j=b(X_i,X_j)\xi$ for a
hypersurface to have
\begin{eqnarray}\label{28}
&&\nabla_{X_i}X_j={\bar\nabla}_{X_i}X_j-b(X_i,X_j)\xi\\\notag
&&=f_{ij}(\nabla_0 f+|\nabla_0 f|^2\partial_0)/(1+|\nabla_0 f|^2).
\end{eqnarray}

On the other hand, it is well known that the Laplacian on a manifold $M^m$ with Riemannian metric $g$ is given by
\begin{eqnarray}\notag
\Delta u=\sum_{i=1}^{m}\{e_ie_i(u)-(\nabla_{e_i}e_i)u\},
\end{eqnarray}
where $\{e_i:i=1, 2, \cdots, m\}$ is an orthonormal frame on $(M^m,g)$. By a straightforward computation one can check that with respect to an arbitrary frame $\{X_i: i=1, 2, \cdots, m\}$ on
$(M,g)$, we have 
\begin{eqnarray}\label{LaA}
\Delta
u=g^{ij}(X_jX_iu-\nabla_{X_i}X_j\,u).
\end{eqnarray}

Using the frame defined in (\ref{Frame}), and Equations (\ref{28}),  (\ref{LaA})  we obtain Formula (\ref{Lap}), and from which we ontain (\ref{Laf})  by replacing $u$ by $f$ and a straightforward computation. Thus, we end the proof of the Claim.\\

Now using (\ref{Lap}) and a straightforward computation we have
\begin{eqnarray}\label{48}
\Delta x_k=-f_k\Delta f, \;\; k=1,2,\ldots, m.
\end{eqnarray}

A further computation using (\ref{Lap}), (\ref{Laf}) and (\ref{48})
shows that (\ref{Bggd}) is indeed equivalent to (\ref{Bg}),  which
ends the proof of the first statement.\\

For statement (II), first notice that the mean curvature of the
graph is given by
\begin{eqnarray}\label{49}
mH=g^{ij}b_{ij}=\frac{g^{ij}f_{ij}}{\sqrt{1+|\nabla_0
f|^2\;}}=\sqrt{1+|\nabla_0 f|^2}\;\Delta f.
\end{eqnarray}
It follows that the function $f$ has minimal graph if and only if
the graph $\phi$, as an isometric embedding, is harmonic. This is
equivalent to the vanishing of both (\ref{Laf}) and (\ref{48}), which
is equivalent to $f$ being a harmonic function with respect to the
metric $\phi^*h_0=g_{ij}=\delta_{ij}+f_if_j$ (and,  by (\ref{49}),
it is equivalent to the mean curvature $H$ vanishes identically).\\

By the implicit function theorem any hypersurface is locally the
graph of a function. It follows that the existence of a non-harmonic
biharmonic hypersurface is equivalent to the existence of a solution
of (\ref{Bg}) which is not harmonic, from which we obtain the
Statement (II) which completes the proof of the theorem.
\end{proof}

\begin{remark}
We notice that the well-known Bernstein's conjecture and Chen's
conjecture on biharmonic hypersurfaces are similar in the following
sense: Bernstein conjecture claimed that for $m\ge 2$, any entire
solution of the minimal graph equation
\begin{equation}\notag
 \Delta f=0 \;\;\;\;\;\;\Longleftrightarrow \;\;\;\;\;\;
\sum_{i,j=1}^{m}\left(\delta_{ij}-\frac{f_{i}f_{j}}{1+\left|\nabla
f\right|^{2}}\right)f_{ij}=0,
\end{equation}
is trivial, i.e., $f$ is an affine function. According to our
Theorem \ref{Conj}, Chen's conjecture for biharmonic hypersurfcaes
is equivalent to claiming that for $m\ge 2$, any solution of the
biharmonic graph equation
\begin{equation}
\begin{cases}
\Delta (\Delta f)=0,\\(\Delta  f_k)\Delta f+2g(\nabla f_k,\nabla
\Delta f)=0, \;\; k=1,2,\ldots, m,
\end{cases}
\end{equation}
is trivial, i.e., $\Delta f=0$.\\

We know that Bernstein's conjecture is true for $m=2$ (Bernstein,
1915),  $m=3$  (De Giorgi,1965), $m=4$ (Almgren, 1966),
 and $m\leq 7$ (J. Simons, 1968); however,  it fails to be true for $m>7$ (Bombieri-De Giorgi-Giusti, 1969).
We are not sure whether the similarity between the two conjectures
implies that Chen's conjecture is more likely to be false. We do
hope that the equivalent analytic formulation of Chen's conjecture
for biharmonic hypersurfaces and its interesting link to Bernstein's
conjecture will attract more work and especially more analytic study
of the conjecture.
\end{remark}
\ack{The author would like to thank the referees for some valuable suggestions.}

\end{document}